\documentclass[11pt,regno]{amsart}
\usepackage{euscript,graphicx,psfrag}
\newtheorem{theorem}{Theorem}
\newtheorem{lemma}{Lemma}

\newtheorem{proposition}{Proposition}
\theoremstyle{definition}
\newtheorem{example}{Example}
\newtheorem{remark}{Remark}
\newtheorem*{remark*}{Remark}
\newtheorem*{defn}{Definition}

\newcommand{\eqdef}{\stackrel{\scriptscriptstyle\rm def}{=}}

\def\bN{\mathbb{N}}

\def\bR{\mathbb{R}}

\def\cM{\EuScript{M}}

\def\cL{\EuScript{L}}

\def\ccL{\widehat\cL}

\DeclareMathSymbol{\varnothing}{\mathord}{AMSb}{"3F}
\renewcommand{\emptyset}{\varnothing}

\author{Katrin Gelfert} \address{Max-Planck-Institut f\"ur Physik komplexer
  Systeme, N\"othnitzer Str. 38, D-01187 Dresden \& Institut f\"ur Physik, TU
  Chemnitz, D-09107 Chemnitz, Germany}
\email{gelfert@pks.mpg.de}
\urladdr{http://www.pks.mpg.de/~gelfert}
\author{{Micha\l} Rams} \address{Institute of Mathematics, Polish Academy of Sciences, ul. \'{S}niadeckich 8, 00-956 Warszawa, Poland}
\email{m.rams@impan.gov.pl}
\urladdr{http://www.impan.gov.pl/~rams}

\begin{document}

\title[]{The Lyapunov spectrum of some parabolic systems}

\begin{abstract}
We study the Hausdorff dimension spectrum for Lyapunov exponents for a class of interval maps which includes several non-hyper\-bolic situations. We also analyze the level sets of points with given lower and upper Lyapunov exponents and, in particular,  with zero lower Lyapunov exponent. We prove that the level set of points with zero exponent has full Hausdorff dimension, but carries no topological entropy.
\end{abstract}

\begin{thanks}
{This research of K.\,G. was supported by the grant EU FP6 ToK SPADE2 and by the Deutsche Forschungsgemeinschaft.
The research of M.\,R. was supported by grants EU FP6 ToK SPADE2, EU FP6 RTN CODY and MNiSW grant 'Chaos, fraktale i dynamika konforemna'.}
\end{thanks}

\maketitle

\section{Introduction}

Our goal here is to present results on the Lyapunov spectrum of interval maps with parabolic periodic points.
We are going to work in the following setting.

Let $f\colon I\to I$ be a map on some interval $I\subset\bR$ for which there is a partition $I=I_1\cup\ldots\cup I_\ell$ into sub-intervals such that $f|I_i$ is monotone and continuously differentiable for every $i$. Let $\Lambda\subset I$ be a compact $f$-invariant set such that $f|\Lambda$ is topologically conjugate to a topologically mixing subshift of finite type. 
Assume that $f|\Lambda$ satisfies the tempered distortion property (see Definition \ref{lem1} for the definition).
Let $(\Lambda_m)_m$ be an increasing family of compact $f$-invariant sets having the property that $f|\Lambda_m$ has bounded distortion and is uniformly expanding and topologically conjugate to subshifts of finite type, and that $\Lambda_m$ converges to $\Lambda$ in the Hausdorff topology.

Our goal is to study the spectrum of Lyapunov exponents of such systems.
Given $x\in\Lambda$ we denote by $\underline\chi(x)$ and $\overline\chi(x)$ the \emph{lower} and \emph{upper Lyapunov exponent} at $x$, respectively, 
\[
\underline\chi(x)\eqdef\liminf_{n\to\infty}\frac{1}{n}\log \lvert(f^n)'(x)\rvert\quad
\overline\chi(x)\eqdef\limsup_{n\to\infty}\frac{1}{n}\log \lvert(f^n)'(x)\rvert,
\] 
and if both values coincide then we call the common value the \emph{Lyapunov exponent} at $x$ and denote it by $\chi(x)$.
For given numbers $0\le \alpha\le\beta$ we consider the following level sets
\[
\cL(\alpha,\beta) \eqdef
 \{x\in \Lambda\colon \underline\chi(x)=\alpha, \, \overline\chi(x)=\beta\}.
\]
If $\alpha<\beta$ then $\cL(\alpha,\beta)$ is contained in the set of so-called \emph{irregular points}
\[
\cL_{\rm irr}\eqdef\left\{ x\in\Lambda\colon \underline\chi(x)<\overline\chi(x)\right\}.
\]
It follows from the Birkhoff ergodic theorem that then we have $\mu(\cL_{\rm irr})=0$ for any $f$-invariant probability measure $\mu$ supported on $\Lambda$. 
We denote by $\cL(\alpha)\eqdef \cL(\alpha,\alpha)$ the set of \emph{regular points} with exponent $\alpha$. 
Similarly, given $0\le\alpha\le\beta$, $\beta>0$ we will study
\[
\widehat\cL(\alpha,\beta) \eqdef
 \{x\in \Lambda\colon 0<\underline\chi(x)\leq\beta, \, \overline\chi(x)\geq\alpha\}.
\]
Recall that the continuous function $\log\,\lvert f'\rvert\colon\Lambda\to\bR$ is said to be \emph{cohomologous} to a constant if there exist a continuous function $\psi\colon\Lambda\to\bR$ and $c\in\bR$ such that $\log\,\lvert f'\rvert=\psi-\psi\circ f+c$ on $\Lambda$, which immediately implies that $\cL(c)=\Lambda$.
From the following considerations we will exclude this trivial case.

We want to determine the complexity of these sets in terms of their Hausdorff dimension $\dim_{\rm H}$. 
The multifractal analysis of dynamical systems, including  level sets of  more general local quantities than the Lyapunov exponents, are so far well understood only in the uniformly hyperbolic  case (see~\cite{Pes:96} for main results and further references). 
Nevertheless, we can mention several results beyond the hyperbolic setting. 
Nakaishi~\cite{Nak:00} studied Manneville-Pomeau-like maps and derived the Hausdorff dimension of the level sets $\cL(\alpha)$ for Lyapunov exponents $\alpha$ in the interior of the spectrum. 
Similar results for a different map was obtained by Kesseb\"ohmer and Stratmann~\cite{KesStr}.

In many approaches to a multifractal analysis  of such level sets one  
characterizes their dimension (or their entropy) in terms of a conditional variational principle of dimension (or entropies) of measures.
We prefer instead a description which involves  the Legendre-Fenchel transform of the pressure function.
To start with our general scheme, it would be desirable to obtain in the above setting a formula for the dimension spectrum of the Birkhoff averages of a \emph{general} continuous (or H\"older continuous) potential $\varphi\colon\Lambda\to\bR$, that is,  to prove for suitable values $\beta$ for example that
\[
\dim_{\rm H}\left\{ x\colon \lim_{n\to\infty}\frac{1}{n} \Big(\varphi(x)+\ldots+\varphi(f^{n-1}(x))\Big) = \beta\right\}
=\frac{1}{\beta} \sup_{d\in\bR}\left(d\beta - P(d\varphi)\right),
\] 
generalizing nowadays classical results (see, e.g.~\cite{BarSau:01}, where, however, the only considered values are in the \emph{interior} of the interval of all the possible  averages $\beta$). 
In the present paper we will investigate the particular case of the potential $\varphi = -\log\lvert f'\rvert$.
Let us denote
\begin{equation}\label{def:Fa}
F(\alpha) \eqdef
\frac{1}{\alpha} \inf_{d\in\bR} \left(P(-d\log\lvert f'\rvert)+\alpha d \right) 
\end{equation}
and let
\begin{equation}\label{def:Fb}
F(0) \eqdef \lim_{\alpha\to 0+} F(\alpha) 
= d_0,
\end{equation}
where 
\[
d_0\eqdef \inf\{d\colon P(-d\log\lvert f'\rvert )=0\}
\]
(see Section~\ref{sec:press} for fundamental properties of $F$).

The following is our first main result.

\begin{theorem} \label{thm:1}
Under the conditions above, for all $0\le \alpha\le \beta$, $\beta>0$, for
which $\ccL(\alpha, \beta)$, $\cL(\alpha, \beta)$ are nonempty we have
\[
\dim_{\rm H} \ccL(\alpha, \beta) = \max_{\alpha\le q\le \beta} F(q)
\quad \text{ and }\quad
\dim_{\rm H} \cL(\alpha, \beta) = \min_{\alpha\le q\le \beta} F(q).
\]
\end{theorem}

The above formulas extend what is known in the hyperbolic setting in several aspects. 
First of all, it applies to several non-hyperbolic situations. 
Second, we are able to cover the boundary points of the Lyapunov spectrum (see~\cite{TakVer:03,PfiSul:07} for related results in the case of the topological entropy of level sets).
Finally, we give a description of the dimension of level sets containing irregular points with zero lower  Lyapunov exponent.
It generalizes results of Barreira and Schmeling \cite{BarSch:00}.

Of particular interest is the set $\cL(0)$. If $f|\Lambda$ satisfies the specification property, then  the entropy spectrum of Birkhoff averages of general continuous potentials have been studied in~\cite{TakVer:03} using a different approach, see also~\cite{PfiSul:07}. 
For such a system, the vanishing of the entropy $h_{\rm top}(f|\cL(0))$ as stated below  follows in fact from~\cite[Theorem 3.5]{TakVer:03} and the Ruelle inequality. 
(Here we note that $\cL(0)$ need not to be compact, and we are using the notion of topological entropy on non-compact sets introduced by Bowen, see Section~\ref{sec:ent}).
On the other hand, in terms of Hausdorff dimension $\cL(0)$ is a rather large set.

\begin{theorem} \label{thm:2}
If $\cL(0)$ is nonempty then we have
\[
\dim_{\rm H} \cL(0) =
 \dim_{\rm H} \Lambda \geq F(0)
\quad\text{ and }\quad
h_{\rm top}(f|\cL(0))=0.
\]
\end{theorem}

We now sketch the exposition of our paper.  In Section~\ref{sec:pre} we review several concepts and results from ergodic theory. 
In Section~\ref{sec:hypsub} we analyze the main properties of the hyperbolic sub-systems which we are going to consider. Upper bounds for the dimension are studied in Section~\ref{sec:upp}. 
Section~\ref{sec:low} is devoted to the analysis of lower dimension bounds: for the set of regular points with an exponent from the interior of the spectrum such bounds simply follow from the maximal lower bound for the corresponding hyperbolic sub-systems. 
In order to handle exponents at the boundary of the spectrum as well as a set of irregular points, we introduce the concept of a w-measure as the main tool of our analysis.
The proofs of Theorems~\ref{thm:1} and~\ref{thm:2} are given at the end of Section~\ref{sec:low} and in Section~\ref{sec:ent}.

\section{Preliminaries}\label{sec:pre}

\subsection{Examples}

Before we collect some examples, let us introduce some notation. 
Consider the topological Markov chain $\sigma\colon \Sigma\to\Sigma$ defined by $\sigma(i_1i_2\ldots)=(i_2i_3\ldots)$ on the set 
\[
\Sigma
\eqdef\{1,\ldots,p\}^\bN
\]
The inverse branches of $\sigma$ will be denoted by $\sigma_i$.
We denote $\Sigma_n=\{1,\ldots,p\}^n$ and $\Sigma_*=\bigcup_{n=0}^\infty \Sigma_n$, where we use the convention $\Sigma_0=\{\emptyset\}$.

We will assume that $f|\Lambda$ is topologically conjugate to a topologically mixing subshift of finite type $(\Sigma_A,\sigma)\subset (\Sigma, \sigma)$.

Let $I_1$, $\ldots$, $I_p$ be a family of compact subintervals of $I$ with pairwise disjoint interiors and assume that $g_i(I)\subset I_i$, where $g_i$ is the inverse branch of $f$, conjugate to $\sigma_i$.
For each $(i_1\ldots i_n)\in\Sigma_n$ we define
\[
\Delta_{i_1\ldots i_n}\eqdef  g_{i_1\ldots i_{n-1}}(I_{i_n})
\]
and $\Delta_\emptyset=I$. Given $x\in \Lambda$, we denote a cylinder $\Delta_{i_1  \ldots i_n}$ containing $x$ also by $\Delta_n(x)$.

Our standing assumption is the tempered distortion property of $f$:

\begin{defn}\label{lem1}
	The map $f$ has \emph{tempered distortion} on $\Lambda$ if there exists a positive sequence $(\rho_n)_n$ decreasing to $0$ such that for every $n$ we have
	\begin{equation}\label{t35}
  	\sup_{(i_1\ldots i_n)}\sup_{x,y\in\Delta_{i_1\ldots i_n}}
  	\frac{\lvert (f^n)'(x)\rvert}{\lvert(f^n)'(y)\rvert} 
	\le e^{n\rho_n}.
	\end{equation}
\end{defn}

We say that $f$ is \emph{uniformly expanding} or \emph{uniformly hyperbolic} on an $f$-invariant compact set $K\subset\Lambda$ if there exists $c>0$ and $\lambda>1$ such that $\lvert (f^n)'\rvert\ge c \lambda^n$ everywhere on $K$. 
There are two main classes of (nonuniformly hyperbolic) examples we can work with.
The first class is closely related to parabolic Cantor sets, introduced in~\cite{Urb:96}.

\begin{example}[parabolic IFS]
Assume that $\lvert f'\rvert>1$ everywhere except a finite set of fixed points $p_i$ where $\lvert f'(p_i)\rvert=1$.
Assume also that $f$ is $C^{1+s}$ for some positive $s$.
We construct the subsystems $\Lambda_m$ by removing some small cylinder neighborhoods of parabolic points and all their pre-images.
Those subsystems are hyperbolic and have bounded distortion.

This class of examples contains for example the celebrated Manneville-Pomeau maps~\cite{PomMan}: $f\colon[0,1]\to[0,1]\colon x\mapsto x(1+x^s)\mod 1$, $s>0$.
\end{example}

\begin{remark}
Strictly speaking, the Manneville-Pomeau map is not conjugated to a subshift of finite type (some cylinders are only essentially disjoint, thus there exists a countable family of points belonging to two different cylinders of the same level).
We will allow this situation, our proofs work in this case as well without major changes.
\end{remark}

The second class is related to the one introduced in~\cite{GelRam:07}.

\begin{example}[expansive Markov systems]
Consider less restrictive assumptions about $f$, demanding only that 
\[
\lim_{n\to\infty} \max_{i_1\ldots i_n}\, \lvert \Delta_{i_1\ldots i_n}\rvert =0 .
\]
Assume also that  $f$ is piecewise $C^2$.
The subsystems $\Lambda_m$ are constructed like in the previous case.
Their hyperbolicity and bounded distortion property follows from the Ma\~{n}\'e hyperbolicity theorem, see~\cite{MelStr:93} for the reference.
\end{example}

\begin{remark}
For both the above-mentioned classes of examples we have equality in the assertion of Theorem \ref{thm:2}.
\end{remark}

\subsection{Topological pressure}\label{sec:press}

Let $\varphi$ be a continuous function on $\Lambda$. The \emph{topological pressure} of $\varphi$ (with respect to $f|\Lambda$) is defined by
\begin{equation}\label{wurm}
  P(\varphi) \eqdef
  \lim_{n\to\infty}\frac{1}{n} \log \sum_{(i_1\ldots i_n)}
      \exp \max_{x\in\Delta_{i_1\ldots i_n}} S_n\varphi(x),
\end{equation}
where here and in the sequel the sum is taken over the cylinders with non-empty intersection with the set $\Lambda$.
The existence of the limit follows easily from the fact that the sum constitutes a sub-multiplicative sequence. 
Moreover, the value $P(\varphi)$ does not depend on the particular Markov partition that we use in its definition.  

Denote by $\cM(\Lambda)$ the family of $f$-invariant Borel probability measures on $\Lambda$. 
We simply write $\cM=\cM(\Lambda)$ if there is no confusion about the system.
By the variational principle we have
\begin{equation}\label{wirm}
 P(\varphi)= 
      \max_{\mu\in\cM}\left( h_\mu(f)+\int_{\Lambda} \varphi d\mu\right),
\end{equation}
where $h_\mu(f)$ denotes the entropy of $f$ with respect to $\mu$ (see~\cite{Wal:81}). 
A measure $\mu\in\cM$ is called \emph{equilibrium state} for the potential $\varphi$ if
\[
P(\varphi)=h_\mu(f)+\int_\Lambda\varphi\,d\mu.
\]

Given $d\in\bR$, we define the function $\varphi_d\colon\Lambda_Q\to\bR$ by
\begin{equation}\label{varphit}
        \varphi_d(x) \eqdef
         -d\log\lvert f'(x)\rvert.
\end{equation}
The tempered distortion property~\eqref{t35} ensures in particular that in the definition of $P(\varphi_d)$ in~\eqref{wurm} one can replace the maximum by the minimum or, in fact, by any intermediate value. 

\begin{proposition}
  The function $d\mapsto P(\varphi_d)$ is a continuous, convex, and non-increasing function of $\bR$. 
$P(\varphi_d)$ is negative for large $d$ if and only if there exist no $f$-invariant probability measures with zero Lyapunov exponent.   
\end{proposition}  

\begin{proof}
  The claimed properties follow immediately from general facts about the
  pressure together with the variational principle~\eqref{wirm}. 
\end{proof}

\begin{lemma}\label{lem:dim}
	We have
	$ d_0 \le \dim_{\rm H}\Lambda$.
\end{lemma}

\begin{proof}
	It follows immediately from the generator condition that every ergodic $f$-invariant measure $\nu$ has non-negative Lyapunov exponent. 	
Suppose now that $P(-t\log\lvert f'\rvert)>0$ for some $t\ge 0$. 
Then, by the variational principle, there exists an ergodic $f$-invariant measure $\nu$ such that $ h_\nu(f)-t\chi(\nu)>0$ and thus $h_\nu(f)>0$. 
It follows then from~\cite{HofRai:92} that 
	\[
	t<\frac{h_\nu(f)}{\chi(\nu)} =
	\dim_{\rm H}\nu\le \dim_{\rm H}\Lambda.
	\]
\end{proof}

We now give a geometric description of $F$ defined in~\eqref{def:Fa},~\eqref{def:Fb} for positive $\alpha$. In general, $F$ is always a concave function with range $\{-\infty\} \cup [0,\infty)$. Let us write $P(d)\eqdef P(-d\log\lvert f'\rvert)$. 
Note that, by continuity and convexity, the pressure function $d\mapsto P(d )$ may fail to be differentiable on an at most countable set. We sketch below the particular case that we may have at most one point of non-differentiability.
If $P'(d)=-\alpha$ then 
\[
F(\alpha)= \frac {P(d)+\alpha d} {\alpha}.
\]
 If $P(d)\geq 0$ for all $d$, that is, the system is parabolic, then we have
\[
F(0) = \inf\{s\geq 0\colon P(s)=0\}.
\] 
In this situation we have the following two possible cases. \\
Case I: The  pressure function is differentiable at $F(0)$. Then $F(\alpha)$ is strictly decreasing for positive $\alpha$.\\
Case II: The  pressure function is not differentiable at $F(0)$. 
Then $F(\alpha)=F(0)$ for every $\alpha \le - \lim_{s\to F(0)_-} P'(s )$, and $F$ is strictly decreasing for greater $\alpha$.

\begin{figure}[h]\label{figuras}
      \psfrag{s1}[t][rr]{$\sim -d\alpha^+$}
      \psfrag{s2}[t][rr]{$\sim-d\alpha^-$}
      \psfrag{smi}[t][rr]{$\alpha^-$}
      \psfrag{spl}[t][rr]{$\alpha^+$}
      \psfrag{p1}[t][ll]{$P(d)$}
      \psfrag{dim}[t][cc]{$\dim_{\rm H}\Lambda$}
      \psfrag{a}[b][rr]{$d$}
      \psfrag{sal}[b][rr]{$\sim-d\alpha_0$}
      \psfrag{al}[t][rr]{$\alpha_0$}
      \psfrag{mi}[b][r]{$-\infty$}
      \psfrag{t1}[t][ll]{$F(\alpha)$}
      \psfrag{d}[t][cc]{$\alpha$}
      \includegraphics[width=0.7\linewidth]{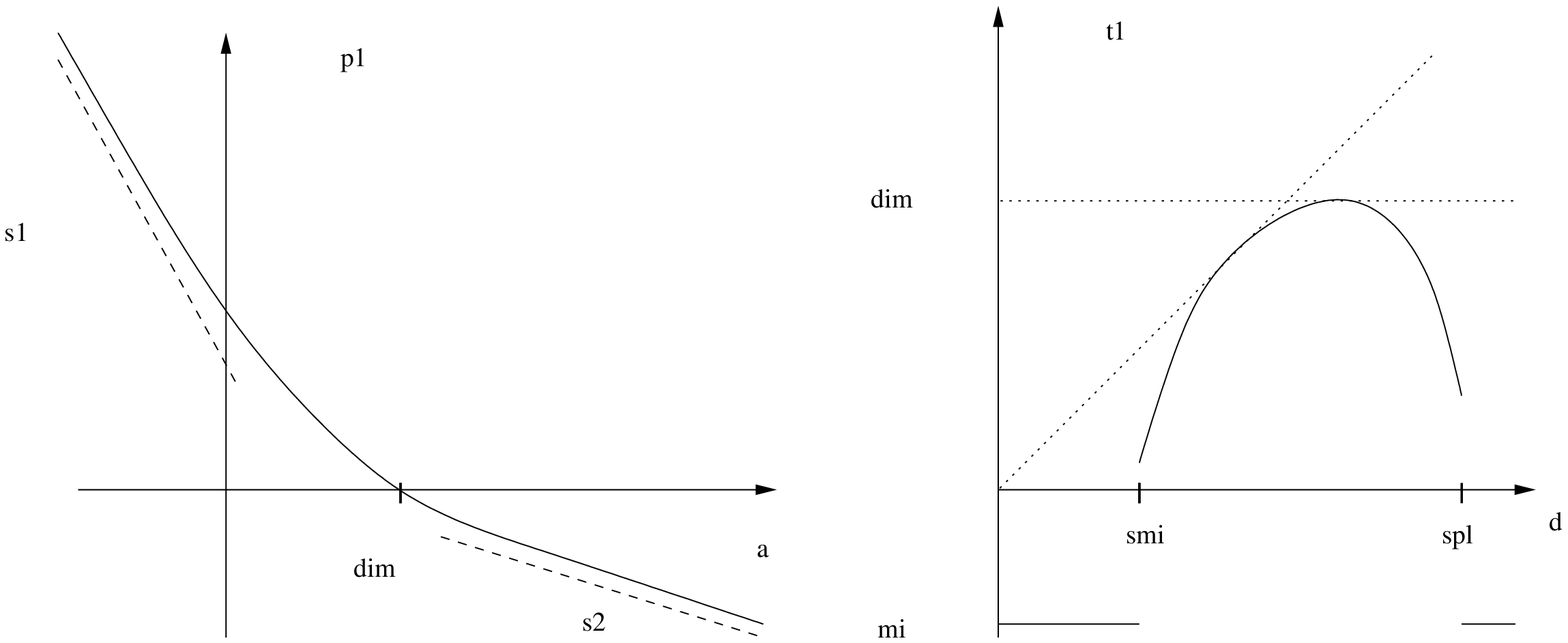} 
  \caption{Pressure and Lyapunov spectrum for uniformly hyperbolic system}
      \includegraphics[width=0.7\linewidth]{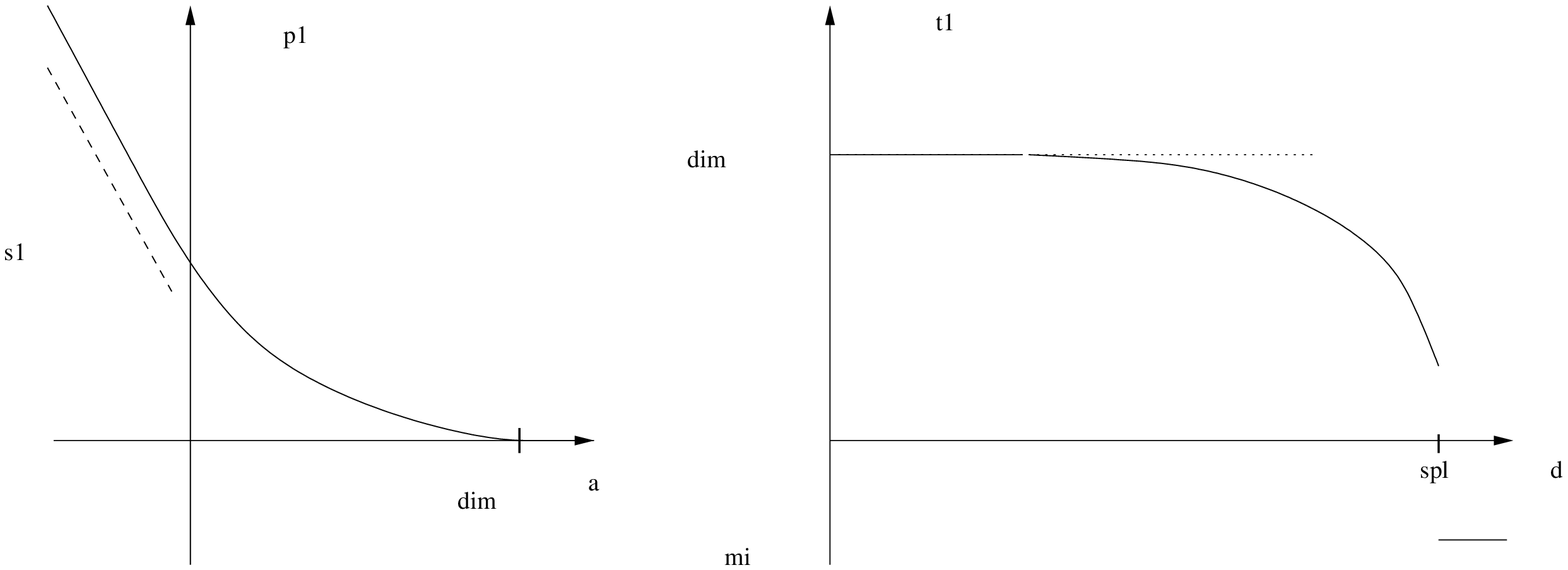} 
   \caption{Pressure and Lyapunov spectrum for parabolic system, Case I}
      \includegraphics[width=0.7\linewidth]{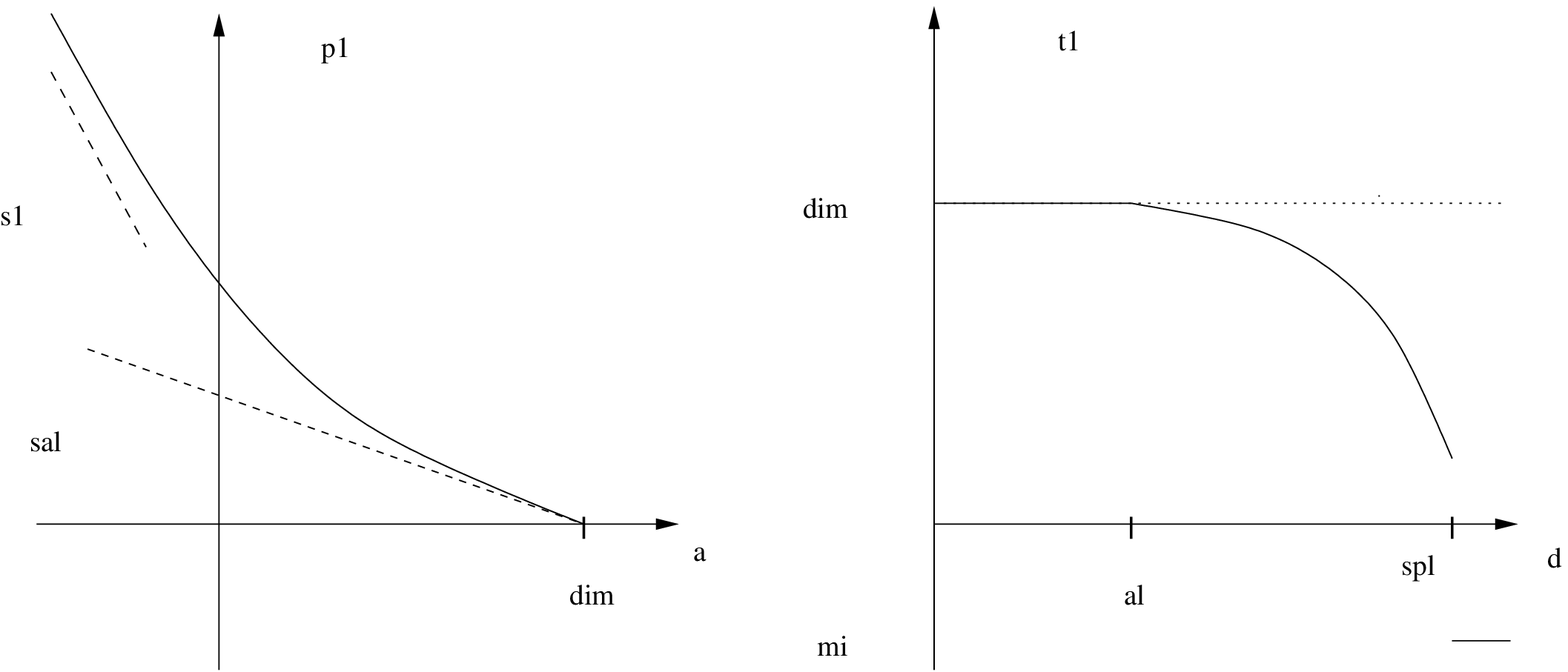} 
    \caption{Pressure and Lyapunov spectrum for parabolic system, Case II}
 \end{figure}

\subsection{Conformal measures}

The \emph{Ruelle-Perron-Frobenius transfer operator} $\cL_\varphi\colon
C(\Lambda)\to C(\Lambda)$ defined on the space $C(\Lambda)$ of continuous
functions $\psi\colon\Lambda\to\bR$ is given by
\[
\cL_\varphi\psi(x) \eqdef \sum_{\tiny \begin{matrix}f(y)=x,\\ y\in \Lambda\end{matrix}}e^\varphi(y)\psi(y)
\]
if $x\in\Lambda$. Denote by $\lambda_\varphi$ the spectral radius of $\cL_\varphi$. Let $\nu$ be an eigenmeasure of the dual operator $\cL_\varphi^\ast$ with eigenvalue $\lambda_\varphi$.
Note that $\nu$ is a probability measure but not necessarily $f$-invariant. However, the dynamical properties of $\nu$ with respect to $f|\Lambda$ are captured through its Jacobian. The \emph{Jacobian} of $\nu$ with respect to $f|\Lambda$ is the (essentially) unique function $J_\nu f$ determined through
\begin{equation}\label{conf}
\nu(f(A)) = \int_AJ_\nu f d\nu.
\end{equation}
for every Borel subset $A$ of $\Lambda$ such that $f|_A$ in
injective, and is given by $J_\nu f = \lambda_\varphi e^{-\varphi}$
(see~\cite{Wal:78}). 
Moreover, by~\cite[Theorem 2.1]{Yur:99} we have $\log\lambda_\varphi = P(\varphi)$.
A measure satisfying~\eqref{conf}  is called \emph{$e^{P(\varphi)-\varphi}$-conformal measure}. 
Such a measure always exists if $f|\Lambda$ is expansive and open (\cite[Theorem 3.12]{DenUrb:91}).

\section{Hyperbolic sub-systems}\label{sec:hypsub}

For shortness, we will write
\[
P_m=P_{f|\Lambda_m}.
\]
Given $m$, the function $d\mapsto P_m(\varphi_d)$ is analytic and strictly decreasing. 
For fixed $d\in\bR$, the sequence $P_m(\varphi_d)$ is non-decreasing.

\begin{proposition}\label{convpress}
	Given $d\in\bR$, we have
	 $P(\varphi_d)=\lim_{m\to\infty}P_m(\varphi_d)$.
\end{proposition}

\begin{proof}
	Assume that this is not the case for some $d\in\bR$. 
	Clearly, $P_m(\varphi_d)$ form an increasing sequence.
	Notice that $P^\ast(d)\eqdef\lim_{m\to\infty}P_m(\varphi_d)\le P(\varphi_d)$. 
Let $\delta\eqdef P(\varphi_d)-P^\ast(d)$. 
Let $m_0\ge 1$ such that $P_m(\varphi_d)\ge P^\ast(d)-\delta/2$ for every $m\ge m_0$.
	
There exists a sequence of $\exp\left(P_m(\varphi_d)-\varphi_d\right)$-conformal measures (with respect to the sub-system $f|\Lambda_m$) which we denote by $\nu_d^m$. 
Each such measure satisfies
	\[
	1= \nu_d^m(f^n(\Delta_n(x))) 
	= \int_{\Delta_n(x)} e^{nP_m(\varphi_d)} 
		\lvert (f^n)'(y)\rvert^d\,d\nu_d^m(y)
	\]
	for every $n\ge 1$ and every $x\in\Lambda_m$. 
	Hence, from the tempered distortion property~\eqref{t35} we can conclude that
	\[
	e^{ -n\rho_n}\le 
	\frac{\nu_d^m(\Delta_n(x))}
		{ \exp\left(-nP_m(\varphi_d)\right) \lvert (f^n)'(x)\rvert^{-d} }\le
	e^{ n\rho_n}.
	\]
	Notice that this inequality holds only for cylinders $\Delta_n(x)$ which intersects $\Lambda_m$. However, if $\Delta_n(x)$ intersects $\Lambda$, then it intersects $\Lambda_m$ for every $m$ sufficiently big.
	
	Likewise for $f|\Lambda$ we obtain for the $\exp\left(P(\varphi_d)-\varphi_d\right)$-conformal measure
	\[
	 e^{ -n\rho_n}\le 
	\frac{\nu_d(\Delta_n(x))}
		{\exp\left(-nP(\varphi_d)\right)\lvert (f^n)'(x)\rvert^{-d} }\le  
	e^{ n\rho_n}
	\]
	for every cylinder $\Delta_n(x)$ intersecting $\Lambda$.
	Hence, we obtain for every $n\ge 1$ and every cylinder $\Delta_{i_1\ldots i_n}(x)$ which intersects $\Lambda_m$. 
	\[
	\nu_d(\Delta_{i_1\ldots i_n}) 
	\le	\nu_d^m(\Delta_{i_1\ldots i_n}) 
	e^{-n(P(\varphi_d)-P_m(\varphi_d))}
		e^ { 2n\rho_n} 
	\le \nu_d^m(\Delta_{i_1\ldots i_n})  e^{-n\delta}e^ { 2n\rho_n} 
	\]
	for every $m\ge 1$.
	Take a subsequence $(\nu_d^{m_k})_k$ converging to some probability measure $\nu_d^\ast$ in the weak$\ast$ topology. Then we obtain
	\[
	\nu_d(\Delta_{i_1\ldots i_n}) \le 
		\nu_d^\ast(\Delta_{i_1\ldots i_n})  e^ { n(2\rho_n-\delta)}
		< \nu_d^\ast(\Delta_{i_1\ldots i_n}) 
	\]	
	for every  $(i_1\ldots i_n)\in\Sigma_{Q_m,n}$. 
	This contradicts the fact that both measures are probability measures. 
	\end{proof}
	
We introduce some further notation. Let 
\[\begin{split}
\alpha_m^- &\eqdef
 \inf\{\alpha\ge 0\colon \chi(x)=\alpha\text{ for some }x\in\Lambda_m\}, \\
\alpha_m^+ &\eqdef
 \sup\{\alpha\ge 0\colon \chi(x)=\alpha\text{ for some }x\in\Lambda_m\} .
\end{split}\]
Similarly, let
\[\begin{split}
\alpha^- &\eqdef
 \inf\{\alpha\ge 0\colon \chi(x)=\alpha\text{ for some }x\in\Lambda\}, \\
\alpha^+ &\eqdef
 \sup\{\alpha\ge 0\colon \chi(x)=\alpha\text{ for some }x\in\Lambda\} .
\end{split}\]
Those are easy to calculate using the pressure, since we have 
\[\begin{split}
\alpha^- &= \lim_{d\to\infty} - \frac 1 d P(\varphi_d), \\
\alpha^+ &= \lim_{d\to -\infty} - \frac 1 d P(\varphi_d).
\end{split}\]

\begin{lemma} \label{lem:apm}
We have
\[\lim_{m\to \infty} \alpha_m^- = \inf_{m\ge 1}\alpha_m^-=\alpha^- ,
\quad
\lim_{m\to \infty} \alpha_m^+ = \sup_{m\ge 1}\alpha_m^+=\alpha^+.
\]
\end{lemma}

\begin{proof}
We have for $d<0$
\[ 
P_m(\varphi_d)+ d \alpha_m^+ \le P_m(0)  \le P(0) , 
\]
hence,  by Proposition~\ref{convpress}, we obtain
\[
\alpha^+ = \lim_{d\to -\infty} - \frac{1}{d} P(\varphi_d) \le \sup_{m\ge 1} \alpha_m^+ .
\]
Similarly, for $d>0$ and $m$ big enough we have
\[
P_m(\varphi_d) \le P_m(0) - d \alpha_m^- \le P(0) - d \alpha_m^-,
\]
and  thus
\[
\alpha^- = \lim_{d\to \infty} - \frac{1}{d} P(\varphi_d) \ge \inf_{m\ge 1} \alpha_m^- .
\]
The opposite inequalities follow from the definition of $\alpha^\pm$ and $\alpha_m^\pm$ and from $\Lambda_m\subset \Lambda$.
\end{proof}

Given $\alpha>0$ and $m\ge 1$ let us denote
\[
	F_m(\alpha)\eqdef 
	\frac{1}{\alpha}\inf_{d\in\bR}\left(P_m(\varphi_d)+\alpha d\right).
\]

\begin{lemma}\label{lem:Fcon}
For every $\alpha\in (\alpha^-, \alpha^+)$ we have
\begin{equation}\label{fcon}
\sup_{m\ge1}F_m(\alpha)
= \lim_{m\to\infty}F_m(\alpha) 
= F(\alpha).
\end{equation} 
\end{lemma}

\begin{proof}
First notice that we can rewrite
\[\begin{split}
F(\alpha)& =
\frac{1}{\alpha}\inf_{d\in\bR}\left(P(\varphi_d)+\alpha d\right)\\
&= \sup \{d\colon d\alpha\le P(\varphi_s) + s\alpha \text{ for every }s\}.
\end{split}\]	
The analogous relation holds for $F_m$ with $P$ replaced by $P_m$.
Let us assume that there exists $\varepsilon>0$ such that $F_m(\alpha) < F(\alpha) - \varepsilon$ for every $m\ge 1$. This would imply that for every $m$ the set
\[
J_m\eqdef \{s\colon P_m(\varphi_s) +s\alpha \leq (F(\alpha) - \epsilon)\alpha\}
\] 
is non-empty, closed, and bounded. 
Moreover, as $P_{m+1}\ge P_m$, we have $J_{m+1}\subset J_m$. Hence, $\bigcap_{m\ge 1} J_m$ is non-empty. 
For $s\in \bigcap_{m\ge 1} J_m$ we conclude that 
\[
P_m(\varphi_s) +s\alpha 
\le (F(\alpha) -\varepsilon)\alpha
\] 
for every $m\ge 1$. Together with Proposition~\ref{convpress} we hence would obtain  
\[
P(\varphi_s)+s\alpha \le (F(\alpha) -\varepsilon) \alpha
\]
which is a contradiction.

We mention a second way of proving~\eqref{fcon} which is based on the convex conjugate functions. 
Let
\[
T_m(\alpha)\eqdef \sup_{d\in\bR}\left(\alpha d-P_m(\varphi_{-d})\right)
\]
denote the convex conjugate of $d\mapsto P_m(\varphi_{-d})$. 
Then $(P_m,T_m)$ form a Legendre-Fenchel pair.
Wijsman~\cite{Wij:66} has shown that for given Legendre-Fenchel pairs $(P_m,T_m)$ and $(P,T)$, the functions $P_m$ converge infimally to $P$ if and only if $T_m$ converges infimally to $T$ (we refer to~\cite{Wij:66} for the definition of infimal convergence). In general, this kind of convergence does not coincide with the pointwise convergence. However, by monotonicity and continuity of the pressure function 
we obtain that $P_m$ converges infimally if and only if it converges pointwise. The application of Proposition~\ref{convpress} implies~\eqref{fcon}.
\end{proof}

For the remainder of this section let $K\subset \Lambda$ be some  $f$-invariant compact set  such that $f|K$ is uniformly expanding.
We have  the following result by Jenkinson~\cite{Jen:01}.
	
\begin{lemma}[\cite{Jen:01}]\label{lem:Jen}
	For any $\alpha\in(\inf_{\nu\in\cM(K)}\chi(\nu),\sup_{\nu\in\cM(K)}\chi(\nu))$  there exists a number  $q=q(\alpha)$ and some equilibrium state $\nu=\nu(\alpha)$ for  the potential $q\log\lvert f'\rvert$  (with respect to $f|K$) such that 
	\[
	\int_K \log\,\lvert f'\rvert \,d\nu=\alpha.
	\]
\end{lemma}

We finally collect results on the dimension of level sets for hyperbolic systems. 

\begin{proposition}\label{proclass}
	For every $\alpha\in(\inf_{\nu\in\cM(K)}\chi(\nu),\sup_{\nu\in\cM(K)}\chi(\nu))$ we have
	\begin{equation}\label{hoh}
	\dim_{\rm H}\left(K\cap \cL(\alpha)\right)
	= \frac{1}{\alpha}
		\inf_{d\in\bR} \left( P_{f|K}(\varphi_d) +d\alpha\right).
	\end{equation}
\end{proposition}

\begin{proof}
	From~\cite[Theorem 6]{BarSau:01} it follows that for arbitrary $d\in\bR$
	\[\begin{split}
	\dim_{\rm H}&\left(K\cap \cL(\alpha)\right)
	=\max\left\{ \frac{h_\nu(f)}{\int_K\log\lvert f'\rvert \,d\nu}
			\colon \nu\in\cM(K), 
				\int_K\log\lvert f'\rvert \,d\nu = \alpha
				\right\}\\
	&=\frac{1}{\alpha}
		\max\left\{ h_\nu(f) + d\int_K\log\lvert f'\rvert\,d\nu
			\colon \nu\in\cM(K), \chi(\nu) = \alpha
				\right\}	-d\\
	&\le\frac{1}{\alpha}
		\max\left\{ h_\nu(f) + d\int_K\log\lvert f'\rvert\,d\nu
			\colon \nu\in\cM(K)\right\}-d\\
	&=\frac{1}{\alpha}\left( P_{f|K}(d\log\lvert f'\rvert) -d\alpha\right),
	\end{split}\]
	where we applied the variational principle for the topological pressure. So we obtain
	\[
	\dim_{\rm H}\left(K\cap \cL(\alpha)\right) \le 
	\frac{1}{\alpha}
	\inf_{d\in\bR}\left( P_{f|K}(\varphi_d) + d\alpha\right).
	\]
	Lemma~\ref{lem:Jen}  implies that
	\[\begin{split}
	\max\left\{ h_\nu(f)
			\colon \nu\in\cM(K), \chi(\nu) = \alpha
				\right\}	
	&\ge 	P_{f|K}(-q\log\lvert f'\rvert) +q\alpha		\\
	&\ge	\inf_{d\in\bR}\left( P_{f|K}(\varphi_d)+d\alpha\right).
	\end{split}\]
	This finishes the proof.
\end{proof}


Given  any H\"older continuous potential $\psi\colon K\to \bR$, there exists a unique ergodic equilibrium state $\mu\in\cM(K)$ which moreover has the Gibbs property, that is, for which there exists a constant $c>1$ such that for all $x\in K$ and every $n\ge 1$ we have
\begin{equation}\label{gibbs}
	c^{-1}\le 
	\frac{\mu(\Delta_n(x)\cap K)}
		{\exp\left(-nP_{f|K}(\psi)+S_n\psi(x)\right)}\le
		c\, .
\end{equation}
We refer for example to~\cite{Pes:96} for more details and references of the above results.
	
\section{Upper bound for the dimension}\label{sec:upp}

\begin{proposition}\label{prop:dimH}
  We have for  every $0\le\alpha\le\beta$, $\beta>0$
	\[
	\dim_{\rm H} \ccL(\alpha,\beta) \le \max_{\alpha\le q \le \beta} F(q) 
	\quad\text{ and }\quad
	\dim_{\rm H} \cL(\alpha,\beta) \le \min_{\alpha\le q \le \beta} F(q) .
  	\] 
\end{proposition}

\begin{proof}
Note first that $\cL(\alpha, \beta)\subset \ccL(q,q)$ for any $q\in [\alpha, \beta]\setminus \{0\}$. 
Hence, the second assertion follows from the first one.

We now prove the first assertion. For a point $x\in \ccL(\alpha, \beta)$ there exists a positive number $q=q(x)\in[\alpha,\beta]$ and a sequence $(n_k)_k$ for which we have  
\begin{equation} \label{eqn:gh}
\lim_{k\to\infty}\frac{1}{n_k}\log\lvert (f^{n_k})'(x)\rvert=q
\end{equation}
Let $\delta\in(0,q)$. There exists $N=N(x)\ge 1$ such that
\begin{equation} \label{est1}
e^{n_k(q-\delta)}\le \lvert(f^{n_k})'(x)\rvert \le e^{n_k(q+\delta)}
\end{equation}
for every $n_k\ge N$. 
By the tempered distortion property~\eqref{t35} we obtain
\begin{equation} \label{est2}
\lvert \Lambda\rvert \,\lvert(f^{n_k})'(x)\rvert^{-1} \, e^{-n_k\rho_{n_k}}
\le \lvert \Delta_{n_k}(x)\rvert 
\le \lvert \Lambda\rvert \,\lvert(f^{n_k})'(x)\rvert^{-1} \, e^{n_k\rho_{n_k}} .
\end{equation}
Using again the tempered distortion property~\eqref{t35} we can conclude that the  $\exp\left(P(\varphi_d)-\varphi_d\right)$-conformal measure satisfies
	\begin{equation}\label{huch}
	 e^{-nP(\varphi_d)} \lvert (f^n)'(x)\rvert^{-d} e^{ -n\rho_n}\le 
	\nu_d(\Delta_n(x))\le
	e^{-nP(\varphi_d)} \lvert (f^n)'(x)\rvert^{-d} e^{ n\rho_n}.
	\end{equation}
We obtain
\begin{equation}\label{jh}
P(\varphi_d) + \lim_{k\to\infty}\frac{1}{n_k}\log\nu_d(\Delta_{n_k}(x)) 
= -d\,q,
\end{equation}
and in particular the limit on the left hand side exists. Hence, possibly after increasing $N$, we have
\[
 e^{-n_k\left( P(\varphi_d) +dq + \delta \right) } \le \nu_d(\Delta_{n_k}(x)) 
\]
for every $n_k\ge N$. 

With~\eqref{est2} and~\eqref{huch}  we can conclude that
\[
\nu_d(\Delta_{n_k}(x))
\ge e^{-n_kP(\varphi_d)} \lvert\Delta_{n_k}(x)\rvert^d
 	\lvert\Lambda\rvert^{-d} \, e^{-n_k\rho_{n_k} (1+\lvert d\rvert)}
	 .
\]
Case 1)
Let us first assume that $P(\phi_d)\geq 0$.
Using~\eqref{est1} and~\eqref{est2} we can estimate 
\[
e^{-n_k P(\varphi_d)}\ge 
\left( \lvert\Delta_{n_k}(x)\rvert\,\lvert\Lambda\rvert^{-1} 
	e^{-n_k\rho_{n_k}}
\right)^{P(\varphi_d)/(q-\delta)}.
\]
Thus we obtain
\[
\nu_d(\Delta_{n_k}(x))\ge
\left(
\lvert\Delta_{n_k}(x)\rvert \, \lvert\Lambda\rvert^{-1} 
\right)^{d+\frac{P(\varphi_d)}{q-\delta}}
	e^{-n_k\rho_{n_k}\left(\frac{P(\varphi_d)}{q-\delta} +(1+\lvert d\rvert)\right)} .
\]
There exists $\varepsilon=\varepsilon(\delta)>0$ such that, perhaps after increasing $N$ again, we have 
\[
\lvert\Lambda\rvert^{-d-\frac{P(\varphi_d)}{q-\delta}}
	e^{-n_k\rho_{n_k}\left(\frac{P(\varphi_d)}{q-\delta} +(1+\lvert d\rvert)\right)}
\ge 	\lvert\Delta_{n_k}(x)\rvert^\varepsilon
\]
for every $n_k\ge N$.
Note that $\Delta_{n_k}(x)\subset B(x,\lvert\Delta_{n_k}(x)\rvert)$.
Hence,  we obtain the following upper bound for the lower pointwise dimension at $x$
\begin{equation} \label{esto3}
\underline{d}_{\nu_d}(x) \le \frac{P(\varphi_d)}{q-\delta}+d+\varepsilon.
\end{equation}
Case 2) Let us now assume that  $P(\varphi_d)< 0$. Using~\eqref{est1} and~\eqref{est2} we can estimate 
\[
e^{-n_k P(\varphi_d)}\ge 
\left( \lvert\Delta_{n_k}(x)\rvert\lvert\Lambda\rvert^{-1} \right)
	^{\frac{P(\varphi_d)}{q+\delta}}
	e^{n_k\rho_{n_k}\frac{P(\varphi_d)}{q+\delta}}.
\]
Thus we obtain
\[
\nu_d(\Delta_{n_k}(x))\ge
\lvert\Delta_{n_k}(x)\rvert^{d+\frac{P(\varphi_d)}{q+\delta}+\varepsilon}
\]
for every $n_k\ge N$, possibly after increasing $N$, and hence in this case
\begin{equation}
 \label{estuo3}
\underline{d}_{\nu_d}(x) \le \frac{P(\varphi_d)}{q+\delta}+d+\varepsilon.
\end{equation}

In both cases, continuity of $d\mapsto P(\varphi_d)$ implies that for any given sufficiently small interval $(q',q'')\subset (\max\{0,\alpha-\delta\}, \beta+\delta)$ there exist $d\in\bR$ such that
\[
\frac{1}{q''} P(\varphi_d) + d \leq F(q'') + \varepsilon .
\]
We can then choose a countable family of intervals $(q_i', q_i'')$, covering $(\max\{0,\alpha-\delta\}, \beta+\delta)$, and consider the corresponding sequence $(d_i)_i$.
Define
\[
\nu \eqdef \sum_{i=1}^\infty 2^{-i} \nu_{d_i}
\]
We have
\[
\underline{d}_\nu(x) \le 
\sup_{i\ge 1} \underline{d}_{\nu_{d_i}}(x) \le
\max_{\alpha-\delta\leq q \leq \beta+\delta} F(q) +2\varepsilon,
\]
where the second inequality follows from~\eqref{esto3},~\eqref{estuo3}.
This implies that 
\[
\dim_{\rm H}\ccL(\alpha, \beta) \le 
\max_{\alpha-\delta\leq q \leq \beta+\delta} F(q) +2\varepsilon.
\]
Since $\delta$ and $\varepsilon$ can be chosen arbitrarily small, this finishes the proof.
\end{proof}

Given $\alpha>0$ we  denote
\begin{equation}\label{hich}
\ccL(\alpha) \eqdef  \left\{x\in \Lambda\colon 
\underline\chi(x)=0, \, \overline\chi(x) \ge \alpha\right\} .
\end{equation}
The following proposition is proved in a similar way to Proposition~\ref{prop:dimH}.

\begin{proposition} \label{prop:ga}
We have for every $\alpha> 0$
\[
\dim_{\rm H}\ccL(\alpha) \le F(\alpha).
\]
\end{proposition}

\section{Lower bound for the dimension}\label{sec:low}

\subsection{The interior of the spectrum -- regular points}

\begin{proposition}\label{prop:lb}
	For $\alpha>\alpha^-$ we have
	\[
	\dim_{\rm H}\cL(\alpha)\ge F(\alpha).
	\]
\end{proposition}

\begin{proof}
	Denote $H_m(\alpha)\eqdef \cL(\alpha)\cap \Lambda_m$. 
For each exponent $\alpha>\alpha^-$ there exists $m\ge 1$ such that $\alpha>\alpha_m^-$ and hence $\alpha>\alpha_{m'}^-$ for every $m'\ge m$. 
By Proposition~\ref{proclass} we have $F_{f|\Lambda_{m'}}(\alpha)=\dim_{\rm H}H_{m'}(\alpha)$, and we can conclude that
	\[
	 \dim_{\rm H}\cL(\alpha)
	 \ge \sup_{m\ge1}\dim_{\rm H}H_m(\alpha)
	 = \sup_{m\ge1} F_{f|\Lambda_m}(\alpha).
	\]
The application of Lemma~\ref{lem:Fcon} finishes the proof.
\end{proof}


\subsection{Construction of w-measures and their properties}\label{sec:wme}

Recall the notation for hyperbolic sub-systems introduced in Section~\ref{sec:hypsub}. 
Given a nondecreasing sequence of positive integers $(n_i)_i$, let $(\mu_i)_i$ be a sequence of certain equilibrium states for potentials $\phi_i$ with respect to $f|\Lambda_{n_i}$.
We denote 
\[
h_i\eqdef h_{\mu_i}(f),\quad
\chi_i\eqdef \chi(\mu_i),\quad 
d_i\eqdef \frac{h_i}{\chi_i}=\dim_{\rm H}\mu_i
.
\]
(Note that the last equality uses a result in~\cite{HofRai:92}.)
We note that the same construction can be performed for an arbitrary, not necessarily non-decreasing, sequence $(n_i)_i$. But this assumption simplifies the exposition.
We will in the following assume that 
\begin{equation}\label{hae}
P_{f|\Lambda_{n_i}}(\phi_i)=0
\end{equation} 
(note that otherwise we can replace $\phi_i$ by $\phi_i-P_{f|\Lambda_{n_i}}(\phi_i)$ without changing the equilibrium state $\mu_i$). 

We now describe the construction of a measure $\mu$, satisfying certain special properties.
Let $(m_i)_i$ be a fast increasing sequence of positive integers. 
We will specify the specific growth speed in the course of this section.
We demand that
\begin{equation} \label{cond05}
\frac {\rho_{m_i}} {\chi_i} \rightarrow 0,
\end{equation}
where $(\rho_m)_m$ is a positive sequence decreasing to $0$ as in~\eqref{t35}.
We define a probability measure $\mu$ on the algebra generated by the cylinders $\Delta_{i_1\ldots i_m}$.
As the beginning of the construction, for cylinders of level $m_1$ we define 
\[
\mu(\Delta_{\omega^{m_1}})\eqdef \mu_1(\Delta_{\omega^{m_1}}).
\]
Given a cylinder of level $m_i$ of positive measure $\mu$, we sub-distribute the measure on its sub-cylinders of level $m_{i+1}$ which intersect $\Lambda_{m_{i+1}}$ in the following way. 
Let
\[
\mu(\Delta_{\omega^{m_i}\tau^{m_{i+1}-m_i}}) 
\eqdef c_{i+1}(\omega^{m_i}) 
        \, \mu(\Delta_{\omega^{m_i}}) 
        \, \mu_{i+1}(\Delta_{\tau^{m_{i+1}-m_i}})
\]
where
\[
c_{i+1}(\omega^{m_i}) = 
\left(
\sum_{\tau^{m_{i+1}-m_i}\colon 
        \omega^{m_i}\tau^{m_{i+1}-m_i} \in \Sigma_{n_{i+1}}} 
\mu_{i+1}(\Delta_{\tau^{m_{i+1}-m_i}})\right)^{-1}
\]
is the normalizing constant. 
For every $m_i<m<m_{i+1}$ let
\[
\mu(\Delta_{\pi^m}) \eqdef
 \sum_{\tau^{m_{i+1}-m}} \mu(\Delta_{\pi^m\tau^{m_{i+1}-m}})\,.
\]
We extend the measure $\mu$ arbitrarily to the Borel $\sigma$-algebra of $\Lambda$.
We call the probability measure $\mu$ a \emph{w-measure} with respect to the sequence $(f|\Lambda_{n_i},\phi_i,\mu_i)_i$. 
We will in the following analyze the precise way in which the Lyapunov exponents and the local entropies of $\mu$ are determined by the asymptotic fluctuations of   the Lyapunov exponents and the local entropies of the equilibrium states $\mu_i$.

As each $\mu_{i+1}$ is an equilibrium state for a uniformly hyperbolic system, it has the Gibbs property~\eqref{gibbs} with some constant $D_{i+1}$. Hence we can conclude that
\begin{multline*}
 D_{i+1}^{-2}\le \\
        \frac{1}{\mu_{i+1}(\Delta_{\omega^{m_i}})}
        \left[ \sum_{\tau^{m_{i+1}-m_i}\colon
         \omega^{m_i}\tau^{m_{i+1}-m_i} \in \Sigma_{n_{i+1}}} 
        \mu_{i+1}(\Delta_{\omega^{m_i} \tau^{m_{i+1}-m_i}}) \right]
        c_{i+1}(\omega^{m_i}) \\
\le D_{i+1}^2.
\end{multline*}
Now observe that $\mu_{i+1}(\Delta_{\omega^{m_i}})^{-1}[\cdots]=1$ and hence
\[
D_{i+1}^{-2} \le c_{i+1}(\omega^{m_i}) \le D_{i+1}^2 .
\]
Notice further that the Gibbs property~\eqref{gibbs} implies that 
\begin{equation}\label{starstarstar}
D_{i+1}^{-3} \le 
\frac{\mu_{i+1}(\Delta_{\omega^{m_i}\tau^{m-m_i}})}
	{\mu_{i+1}(\Delta_{\omega^{m_i}})\mu_{i+1}(\Delta_{\tau^{m-m_i}})}
	\le D_{i+1}^3.
\end{equation}
Hence we obtain for the constructed measure $\mu$
 \begin{equation} \label{ki}
D_{i+1}^{-5}
\frac {\mu(\Delta_{\omega^{m_i}})} {\mu_{i+1}(\Delta_{\omega^{m_i}})}  
\le \frac{\mu(\Delta_{\omega^{m_i} \tau^{m_{i+1}-m_i}})} 
	{\mu_{i+1}(\Delta_{\omega^{m_i} \tau^{m_{i+1}-m_i}})} 
\le D_{i+1}^5\frac {\mu(\Delta_{\omega^{m_i}})} {\mu_{i+1}(\Delta_{\omega^{m_i}})}\,.
\end{equation}
Consider now the measure of a cylinder at level $m$ for any $m_i<m<m_{i+1}$. 
Notice that
\[
\mu_{i+1}(\Delta_{\pi^m}) = 
\sum_{\tau^{m_{i+1}-m}} \mu_{i+1}(\Delta_{\pi^m\tau^{m_{i+1}-m}}) \,.
\]
Hence,~\eqref{ki} implies that for any  $m_i<m<m_{i+1}$
\begin{equation}\label{locu}
D_{i+1}^{-5} \frac {\mu(\Delta_{\omega^{m_i}})} {\mu_{i+1}(\Delta_{\omega^{m_i}})}  \le 
\frac{\mu(\Delta_{\omega^{m_i} \tau^{m-m_i}})} 
	{\mu_{i+1}(\Delta_{\omega^{m_i} \tau^{m-m_i}})} \le 
D_{i+1}^5 \frac {\mu(\Delta_{\omega^{m_i}})} {\mu_{i+1}(\Delta_{\omega^{m_i}})} \,.
\end{equation}
Now~\eqref{locu} and~\eqref{starstarstar} imply
and
\begin{equation} \label{ki2}
D_{i+1}^{-8} \le 
\frac {\mu(\Delta_{\omega^{m_i} \tau^{m-m_i}})} {\mu(\Delta_{\omega^{m_i}}) 
\mu_{i+1}(\Delta_{\tau^{m-m_i}})} \le 
D_{i+1}^8
\end{equation}
for any $m_i<m\le m_{i+1}$ and any sequence $\omega^{m_i} \tau^{m-m_i}$ such that $\Delta_{\omega^{m_i} \tau^{m-m_i}}$ intersects $\Lambda_{i+1}$.

Denote
\[
L_m(x) \eqdef \frac 1 m \log \lvert (f^m)'(x)\rvert 
\]
and
\[
H_m(x) \eqdef - \frac 1 m \log \mu(\Delta_m(x)) .
\]
Note that for $m_i<m\le m_{i+1}$ we have
\begin{equation} \label{eqn:l}
L_m(x) = \frac {m-m_i} m L_{m-m_i}(f^{m_i}(x)) + \frac {m_i} m L_{m_i}(x).
\end{equation}
Further, by~\eqref{ki2} for $m_i<m\le m_{i+1}$ we can estimate
\begin{equation} \label{eqn:h}
\left\lvert H_m(x) + \frac{1}{m} \log \mu_{i+1}(\Delta_{m-m_i}(f^{m_i}(x))) 
	- \frac {m_i} m H_{m_i}(x)\right\rvert
\le \frac {8\log D_{i+1}} {m_i} .
\end{equation}

\begin{proposition} \label{cor1}
	For any $\varepsilon>0$, denote 
	\begin{multline} \label{cond2}
	A_H(\varepsilon, m_i)\eqdef
	 \Big\{ x\colon
	\left\lvert H_m(x) - \frac{m_i}{m} H_{m_i}(x) - \frac{m-m_i}{m} h_{i+1}\right\rvert  
	\le \varepsilon \, \lvert h_{i+1}-h_i\rvert\\
	 \text{ for every }m_i<m\le m_{i+1}\Big\}.
	\end{multline}
	Then for any $\delta>0$ there exists $M(\delta)\ge1$ such that for $m_i>M(\delta)$
	\[
	\mu(A_H(\varepsilon, m_i))> 1- \delta
	\]
\end{proposition}

\begin{proof}
Fix some $\varepsilon>0$ and $\delta>0$.
We will prove that the set $A_H(\varepsilon, m_i)\cap \Delta_{\omega^{m_i}}$ has measure greater than $1-\delta \,\mu(\Delta_{\omega^{m_i}})$ for all $\omega^{m_i}$ when $m_i$ is big enough.
Denote 
\[
C(\varepsilon, m_i, \omega_i) =
 f^{m_i}(A_H(\varepsilon, m_i)\cap \Delta_{\omega^{m_i}}).
\]
As this set is a union of cylinders of level at most $m_{i+1}-m_i$, by~\eqref{ki2} it is enough to prove that $\mu(C(\varepsilon, m_i, \omega_i)) \ge 1-D_{i+1}^{-8} \delta$ uniformly in $\omega^{m_i}$ for $m_i$ big enough.

By~\eqref{eqn:h}, we obtain
\begin{multline} \label{eqn:mid}
	\left\lvert H_m(x) - \frac{m_i}{m} H_{m_i}(x) 
	- \frac{m-m_i}{m} h_{i+1}\right\rvert \\
	\le \frac {8\log D_{i+1}}{m_i} 
	+ \frac{m-m_i}{m} 
	\left\lvert h_{i+1} + \frac{1}{m-m_i} \log \mu_{i+1} (\Delta_n(y))\right\rvert 
\end{multline}
for $y=f^{m_i}(x)$.
From the Gibbs property of the measure $\mu_{i+1}$ we obtain that 
\[
\left\lvert \log\mu_{i+1}(\Delta_n(y)) -S_n\phi_{i+1}(y)\right\rvert \le \log D_{i+1} 
\]
for any $n\ge 1$.
Thus, the right hand side of~\eqref{eqn:mid} is not greater than
\[
W\eqdef
\frac {9\log D_{i+1}} {m_i} + 
\frac{1}{m} \left\lvert (m-m_i)h_{i+1} + S_{m-m_i} \phi_{i+1}(y) \right\rvert .
\]
The first summand is arbitrarily small for big $m_i$.
To estimate the second one we note that it is a consequence of the Birkhoff ergodic theorem and the Egorov theorem that for the given numbers $\varepsilon>0$ and $\delta>0$ there exists $N=N(\varepsilon,\delta)$ such that we have 
\[
\mu_{i+1}\left(\left\{x\colon 
\left\lvert S_n\phi_{i+1}(x) - n \int\phi_{i+1}\,d\mu_{i+1}\right\rvert  
\le n \varepsilon \,\,\,\,\forall n\ge N\right\}\right) \ge 1- \delta.
\]
From~\eqref{hae} and from the fact that $\mu_{i+1}$ is an equilibrium state we conclude that 
\[
h_{i+1} 
= -\int\phi_{i+1}\,d\mu_{i+1}.
\]

Hence, for $m_i$ sufficiently big, $W$ is smaller than any constant with arbitrarily big probability.
In particular, for $m_i>N(\varepsilon\lvert h_{i+1} - h_i\rvert, D_{i+1}^{-8} \delta)$ it is smaller than $\varepsilon\lvert h_{i+1} - h_i\rvert$ with probability bigger than $1- D_{i+1}^{-8} \delta$ and the assertion follows.
\end{proof}

\begin{proposition} \label{cor2}
	For any $\varepsilon>0$, denote
	\begin{multline} \label{cond3}
	 A_L(\varepsilon, m_i)\eqdef
	 \Big\{ x \colon
	 \left\lvert  L_m(x) - \frac{m_i}{m} L_{m_i}(x) - \frac{m-m_i}{m} \chi_{i+1}\right\rvert  
	\le \varepsilon \, \lvert \chi_{i+1}-\chi_i\rvert \\
	\text{ for every }m_i<m\le m_{i+1}\Big\}.
	\end{multline}
	Then for any $\delta>0$ there exists $M(\delta)\ge1$ such that for $m_i>M(\delta)$
	\[
	\mu(A_L(\varepsilon, m_i))\ge 1-\delta.
	\]
\end{proposition}

\begin{proof}
We note that if $y \in \Delta_m(x)$ then $\lvert L_m(x)-L_m(y)\rvert  \le \rho_m$, where $(\rho_m)_{m\ge 1}$ is the to $0$ decreasing sequence from the tempered distortion property~\eqref{t35}.
Now we can apply~\eqref{cond05} and repeat the same reasoning as in the proof of Proposition~\ref{cor1}.
\end{proof}

For any sequence $(\varepsilon_i)_{i\ge1}$ we can choose a summable sequence $(\delta_i)_{i\ge1}$ and a sequence $(m_i)_{i \ge 1}$ such that Propositions~\ref{cor1} and~\ref{cor2} hold for the constructed measure $\mu$.
In such a situation, by the Borel-Cantelli lemma for $\mu$-almost every $x\in \Lambda$ both~\eqref{cond2} and~\eqref{cond3} are satisfied for all except finitely many $i$.
This leads us to the following proposition.

\begin{proposition} \label{prop:lim}
If $(m_i)_i$ in the construction above increases sufficiently fast, then for $\mu$-almost every $x\in \Lambda$ we have
\begin{itemize}
\item[i)] $\displaystyle \liminf_{m\to\infty} H_m(x) = \liminf_{i\to\infty} h_i$,   
	 $\displaystyle \limsup_{m\to\infty} H_m(x) = \limsup_{i\to\infty} h_i$,
\item[ii)] $\displaystyle\liminf_{m\to\infty} L_m(x) = \liminf_{i\to\infty} \chi_i$, 
	 $\displaystyle\limsup_{m\to\infty} L_m(x) = \limsup_{i\to\infty} \chi_i$, and
\item[iii)] $\displaystyle\liminf_{m\to\infty} \frac {H_m(x)} {L_m(x)} = \liminf_{i\to\infty} d_i$.
\end{itemize}
Moreover,
\[
\dim_{\rm H}\mu \ge \liminf_{i\to\infty} d_i .
\]
\end{proposition}

\begin{proof}
Choose some $\varepsilon>0$.
Denote by
\[
A_j\eqdef 
\bigcap_{i\ge j} \left( A_H(\varepsilon_i,m_i)\cap A_L(\varepsilon_i,m_i) \right).
\]
the set of points for which~\eqref{cond2} and~\eqref{cond3} are satisfied for all $i\ge j$ for some sequence $(\varepsilon_i)_{i\ge1}$ (which will be specified in the following).
Note that $\{A_j\}$ is an increasing family of sets and that $\bigcup_{j\ge1}A_j$ is of full measure $\mu$.

Clearly $H_{m_1}(x) \in [H^-, H^+]$ for some $0<H^-$, $H^+<\infty$ independent of $x$ (because there are only finitely many cylinders) and
$L_{m_1}(x) \in [L^-, L^+]$ for some $0<L^-$,  $L^+<\infty$ independent of $x$ (because $\log \,\lvert f'\rvert $ is uniformly bounded on $\Lambda$).

Let $j\ge1$ be such that $\mu(A_j)>0$. 
Given $x\in A_j$, by~\eqref{cond2} we have
\begin{multline*}
\left\lvert H_{m_{j+k+1}}(x)-h_{j+k+1}\right\rvert  \\
\le \left(\frac {m_{j+k}} {m_{j+k+1}} 
	+ \varepsilon_{j+k}\right) \lvert h_{j+k+1} - h_{j+k}\rvert  
	+ \frac {m_{j+k}} {m_{j+k+1}} \lvert H_{m_{j+k}}(x)-h_{j+k}\rvert ,
\end{multline*}
and hence for every $\ell\ge j$ we obtain
\[
\left\lvert \frac{H_{m_{\ell+1}}(x)}{h_{\ell+1}}-1\right\rvert  
\le \left(\frac {m_\ell} {m_{\ell+1}} 
	+ \varepsilon_\ell\right) \left\lvert \frac{h_\ell}{h_{\ell+1}} - 1 \right\rvert  
	+ \frac {m_\ell} {m_{\ell+1}} \left\lvert \frac{H_{m_\ell}(x)}{h_\ell}- 1 \right\rvert 
		\frac{h_\ell}{h_{\ell+1}} .
\]
Thus, if $\varepsilon_i$ is sufficiently small 
and if $m_i$ grows fast enough then we obtain that
\[
\frac {H_{m_i}(x)} {h_i} \rightarrow 1
\]
as $i\to \infty$ uniformly in $x\in A_j$.

Let $I=I(j)$ be sufficiently big such that for all $x\in A_j$ and all $i>I$ we have
\[
\frac {H_{m_i}(x)} {h_i} \in (1-\varepsilon, 1+\varepsilon).
\]
By~\eqref{cond2},  for all $m_i<m\le  m_{i+1}$ we have then
\begin{equation} \label{res1}
\left\lvert H_m(x) - \frac {m_i} m h_i - \frac{m-m_i}{m} h_{i+1}\right \rvert  
\le  \varepsilon_i \,\lvert h_{i+1}-h_i\rvert  + \varepsilon .
\end{equation}
Using~\eqref{cond3} instead of~\eqref{cond2}, we can, possibly after changing $(\varepsilon_i)_i$ and $(m_i)_i$, prove in an analogous way that for all $m_i<m\le  m_{i+1}$ we have
\begin{equation} \label{res2}
\left\lvert L_m(x) - \frac{m_i}{m} \chi_i - \frac{m-m_i}{ m} \chi_{i+1}\right\rvert
  \le \varepsilon_i \,\lvert \chi_{i+1}-\chi_i\rvert  + \varepsilon. 
\end{equation}
Notice that
\[
 h(m)\eqdef\frac{m_i}{m} h_i + \frac{m-m_i}{m} h_{i+1} 
\]
satisfies $h_i \le h(m)\le h_{i+1}$, which implies claim i) of the assertion.
Claim ii) follows from~\eqref{res2} in an analogous way together with
\[
 \chi(m)\eqdef\frac{m_i}{m} \chi_i + \frac{m-m_i}{m} \chi_{i+1}.
\] 
satisfying $\chi_i\le \chi(m) \le \chi_{i+1}$.
As 
\[ 
\frac{h(m)}{\chi(m)} \ge \min\{d_i, d_{i+1}\}
\]
for all $m_i<m\le m_{i+1}$, claim iii) of the assertion follows from~\eqref{res1} and~\eqref{res2}.

We finally prove the lower bound on the Hausdorff dimension of $\mu$.
For all $x\in A_j$ and for all $i\ge I$ we have
\[
H_{m_i}(x) \ge (1-\varepsilon) h_i
\]
and
\[
L_{m_i}(x) \le (1+\varepsilon) \chi_i.
\]
Let $\widetilde\mu_j$ be the restriction of $\mu$ to $A_j$. 
For all $x\in A_j$ we have
\[
\widetilde\mu_j(\Delta_m(x)) \le 
\mu(\Delta_{m_i}(x)) 
= e^{- m_i H_{m_i}(x)}
\le e^{-m_i(1-\varepsilon) h_i}
\]
and
\[
\lvert \Delta_{m_i}(x)\rvert  \ge 
e^{-m_i \rho_{m_i}}e^{-m_i L_{m_i}(x)} \ge 
e^{-m_i (1+2\varepsilon) \chi_i}
\]
for $m_i$ big enough, where we use~\eqref{cond05} to obtain the second inequality.
Let 
\[
r_i \eqdef e^{-m_i (1+2\varepsilon) \chi_i}.
\] 
We obtain
\[
\widetilde\mu_j(B(x,r_i)) \le 2 e^{-m_i(1-\varepsilon) h_i},
\]
and hence the lower pointwise dimension of $\widetilde\mu_j$ at $x$ is bounded by
\[
\underline d_{\widetilde\mu_j}(x) \ge (1-3\varepsilon) \liminf_{i\to\infty} d_i.
\]
Since $\varepsilon$ was arbitrary, we obtain $\underline d_{\widetilde\mu_j}(x)\ge \liminf_{i\to\infty} d_i$ for every $x\in A_j$. Thus, we can conclude that $\dim_{\rm H} \mu \ge\liminf_{i\to\infty} d_i$.
\end{proof}

\subsection{The boundary of the spectrum -- regular and irregular points}

The considerations in the previous section prove the following theorem.

\begin{theorem} \label{thm:meas}
Let $(n_i)_i$ be a nondecreasing sequence of positive integers and $(\mu_i)_i$ be a sequence of equilibrium states for $f|\Lambda_{n_i}$. 
Then 
\[
\dim_{\rm H}\left\{ x\colon \underline\chi(x) = \liminf_{i\to\infty}\chi(\mu_i),\quad
		\overline\chi(x)=\limsup_{i\to\infty}\chi(\mu_i) \right\} 
\ge \liminf_{i\to\infty} \dim_{\rm H}\mu_i.		
\]
\end{theorem}

We now are able to bound the dimension of irregular points from below. 
We also obtain bounds on the dimension of the level sets at the boundary of the spectrum.

\begin{proposition} \label{prop:ha}
	For $\alpha^-\le\alpha\le\beta\le \alpha^+$ we have
	\[
	\dim_{\rm H}\ccL(\alpha,\beta) \ge \max_{\alpha\leq q\leq \beta} F(q)
	\]
and
	\[
	\dim_{\rm H}\cL(\alpha,\beta) \ge \min_{\alpha\leq q\leq \beta} F(q).
	\]
	In particular, we have
	\[
	\dim_{\rm H}\cL(\alpha)\ge F(\alpha).
	\]
\end{proposition}

\begin{proof}

	We consider some sequence $(a_n)_{n\ge n_0}$ of numbers $\alpha_n^-< a_n< \alpha_n^+$ such that $\lim_{n\to\infty}a_n=\alpha$ (by our assumptions, $\alpha_n^- < \alpha_n^+$ for $n$ big enough).  
By Lemma~\ref{lem:Jen}, there exists a sequence $(q_n)_{n\ge1}$ of numbers and a sequence $(\nu_n)_{n\ge1}$ of equilibrium states of the potentials $q_n\log\lvert f'\rvert$ such that
	\[
	P_{f|\Lambda_n}(q_n\log\lvert f'\rvert) = h_{\nu_n}(f)+q_na_n.
	\] 
If we apply~\cite[Theorem 1]{HofRai:92} to each of the hyperbolic sub-systems $f|\Lambda_n$ we obtain that
	\[
	\dim_{\rm H}\nu_n=
	\frac{1}{a_n}\left(P_{f|\Lambda_n}(q_n\log\lvert f'\rvert) -q_na_n\right) \ge
	F_{f|\Lambda_n}(a_n)	
	\]
Note that continuity of the function $F$ and Lemma~\ref{lem:Fcon} together imply that $\lim_{n\to\infty}F_{f|\Lambda_n}(a_n)=F(\alpha)$ for properly chosen $(a_n)$. 
The application of  Theorem~\ref{thm:meas} finishes the proof.
\end{proof}

\begin{proof}[Proof of Theorem~\ref{thm:1}]
The assertions follow from Proposition~\ref{prop:ha} and Proposition~\ref{prop:dimH}.
\end{proof}

\begin{proof}[Proof of  Theorem~\ref{thm:2}]
Recall the definition of $\ccL(\alpha)$ in~\eqref{hich}.
To prove the first part of assertion we  use the fact that 
\[
\Lambda\setminus \cL(0) = \ccL(0,\alpha^+) \cup \bigcup_{n\in \bN} \ccL(2^{-n} \alpha^+) 
\]
If $\cL(0)\neq \emptyset$, then   $q\mapsto F(q)$ is a non-increasing function and  hence Propositions~\ref{prop:dimH} and~\ref{prop:ga} imply that the Hausdorff dimension of $\Lambda\setminus \cL(0)$ is not greater than $F(0)$.
At the same time, by Proposition~\ref{prop:ha} we have $\dim_{\rm H} \cL(0) \geq F(0)$.
This implies that $\dim_{\rm H}\Lambda=\dim_{\rm H}\cL(0)\ge F(0)$. 

The second part of the assertion of Theorem~\ref{thm:2} we will prove in the following section.
\end{proof}

\section{Topological entropy}\label{sec:ent}

We first briefly recall one more concept from the thermodynamic formalism
(for a detailed account we refer to~\cite{Pes:96}). Given a set $Z\subset
\Lambda$ (not necessarily compact nor $f$-invariant) and numbers
$\varepsilon>0$, $n\in \bN$, we denote by $M_\varepsilon(Z,n)$ the maximal
cardinality of a set of points in $Z$ which belong to a
$(n,\varepsilon)$-separated set in $\Lambda$. We define the \emph{upper
  capacitive topological entropy} of $f$ on $Z$ by 
\begin{equation}\label{limi}
\overline{Ch}(f|Z) =\lim_{\varepsilon\to 0}\limsup_{n\to\infty}
\frac{1}{n}\log M_\varepsilon(Z,n).
\end{equation}
Analogously, we define the lower capacitive topological entropy of $f$ on $Z$,
denoted by $\underline{Ch}(f|Z)$, by replacing the limes superior
in~\eqref{limi} with the limes inferior. 
Since $f|\Lambda$ is expansive, it in fact
suffices for sufficiently small $\varepsilon$ to take in~\eqref{limi} only the
limit in $n$. 
We denote by $h(f|Z)$ the \emph{topological entropy} of $f$ on $Z$, but we
refer the reader to~\cite{Pes:96} for its precise definition.
The following properties hold:
\begin{itemize}
\item[1.] $ \underline{Ch}(f|Z_1) \le \underline{Ch}(f|Z_2)$ and
  $\overline{Ch}(f|Z_1) \le \overline{Ch}(f|Z_2)$ whenever $Z_1\subset
  Z_2\subset \Lambda$,
\item [2.]  $h(f|Z)\le \underline{Ch}(f|Z) \le \overline{Ch}(f|Z)$,
\item [3.]  for a countable union $Z=\bigcup_{k\in{\mathcal I}}Z_i$ we have
  $h(f|Z)=\sup_{k\in{\mathcal I}}h(f|Z_i)$.
\end{itemize}
Moreover, when $Z\subset\Lambda$ is $f$-invariant and compact then we have coincidence with
the classical topological entropy with respect to $f|Z$, that is,
\begin{equation*}
 h(f|Z) = \underline{Ch}(f|Z) = \overline{Ch}(f|Z) .
\end{equation*}

Recall that $\Delta_n(x)$ denotes the cylinder $\Delta_{i_1 \ldots i_n}$ containing $x$.

\begin{lemma}\label{lem:ftilde}
  $\cL(0)=\{x\in\Lambda\colon
  \limsup_{n\to\infty}\frac{1}{n}\log \lvert\Delta_n(x)\rvert =0\}$.
\end{lemma}

\begin{proof}
  Given $x\in\Lambda$, we have $f^n(\Delta_n(x))\supset I_i$ for some $i\in\{1,\ldots,p\}$. 
  By the mean value theorem there exists $y=y(n)\in\Delta_n(x)$ such that
  $\lvert (f^n)'(y)\rvert\,\lvert\Delta_n(x)\rvert\ge I_i$ and hence
  \[
   \lvert I\rvert^{-1}\le       
  \lvert\Delta_n(x)\rvert^{-1} \le
  \lvert I_i\rvert^{-1}  \lvert(f^n)'(x)\rvert  \frac{\lvert(f^n)'(y)\rvert}{\lvert (f^n)'(x)\rvert}.
  \]
  By the tempered distortion property we obtain
 \[
  \limsup_{n\to\infty}\frac{1}{n}\log\sup_{y\in\Delta_n(x)}
  \frac{\lvert(f^n)'(y)\rvert}{\lvert(f^n)'(x)\rvert} = 0.
  \]
  from here the statement follows.
\end{proof}

\begin{proposition}\label{pro:ent}
  We have $h(f|\cL(0)) = 0$.   
\end{proposition}

\begin{proof}
  We are going to show the existence of a decreasing sequence of sets
  $\cL_k$, all containing $\cL(0)$, such that $\overline{Ch}(f|
  \cL_k)$ decreases to $0$. 

  Given $\varepsilon>0$ and $N\in\bN$ we define the set
  \[
  L_{\varepsilon,N}=
  \{x\in\Lambda\colon \lvert\Delta_n(x)\rvert \ge(1+\varepsilon)^{-n}
        \text{ for every }n\ge N\}\,.
  \]
  Notice that $L_{\varepsilon,N}\subset L_{\varepsilon,N'}$ for $N\le N'$ and
  that by Lemma~\ref{lem:ftilde}
  \[
 \cL(0) \subset \bigcap_{\varepsilon>0}\bigcup_{N\in\bN}L_{\varepsilon,N}.
  \]
  For every $x\in L_{\varepsilon,N}$ we have $\lvert\Delta_n(x)\rvert \ge
  (1+\varepsilon)^{-n}$ for every $n\ge N$. Hence, the number of
  $(n,\varepsilon)$-separated sets needed to cover the set $L_{\varepsilon,N}$
  is at most $ I (1+\varepsilon)^n$. From the definition of
$\overline{Ch}$ we obtain $\overline{Ch}(f|L_{\varepsilon,N})\le
  \log(1+\varepsilon)$. Now it follows that 
  \[
  h(f| \cL(0)) \le 
  h(f|\bigcup_{N\in\bN}L_{\varepsilon,N}) = 
  \sup_{N\in\bN}h(f|L_{\varepsilon,N}) \le 
  \sup_{N\in\bN}\overline{Ch}(f|L_{\varepsilon,N}) \le \log(1+\varepsilon)\,.
  \]
  Since $\varepsilon$ is arbitrary, we can conclude $h(f| \cL(0))=0$. 
\end{proof}

\bibliographystyle{amsplain}




\end{document}